\documentclass{amsart}
\usepackage{graphicx,amsmath,amssymb,amsfonts,amsthm}
\usepackage[all,cmtip]{xy}
\usepackage{hyperref}
\newtheorem{thm}{Theorem}
\newtheorem{cor}[thm]{Corollary}
\newtheorem{lem}[thm]{Lemma}

\newtheorem{alg}[thm]{Algorithm}
\theoremstyle{definition}

\newtheorem{ex}[thm]{Example}
\theoremstyle{remark}
\newtheorem{rem}[thm]{Remark}
\numberwithin{equation}{section}

\newcommand{\FF}{\mathbb{F}}

\newcommand{\Spec}{ \mathrm{Spec}}

\newcommand{\idealin}{\vartriangleleft}

\newcommand{\ord}{\mathrm{ord}}

\newcommand{\Sch}{\mathsf{Sch}}

\newcommand{\OO}{\mathcal{O}}

\newcommand{\ZZ}{\mathbb{Z}}

\newcommand{\GL}{\mathrm{GL}}
\newcommand{\NN}{\mathbb{N}}
\renewcommand{\AA}{\mathbb{A}}
\newcommand{\Ad}{\mathrm{Ad}}
\newcommand{\Aut}{\mathrm{Aut}}

\newcommand{\Gr}{\mathrm{Gr}}
\newcommand{\at}{\tilde{a}}
\newcommand{\psit}{\widetilde{\psi}}
\newcommand{\ft}{\widetilde{f}}
\newcommand{\gt}{\widetilde{g}}

\newcommand{\At}{\widetilde{A}}
\renewcommand{\mod}{\ \mathrm{mod} \ }

\newcommand{\lift}{\mathrm{lift}}
\newcommand{\id}{\mathrm{id}}

\title{Automorphisms of the Affine Line over Non-Reduced Rings}
\author{Taylor Dupuy}

\begin{document}

\begin{abstract}
The affine space $\AA^1_B$ only has automorphisms of the form $aT+b$ when $B$ is a domain. In this paper we study univariate polynomial automorphisms over non-reduced rings $B$. Geometrically these groups appear naturally as transition maps of affine bundles in arithmetic geometery.
\end{abstract}

\maketitle

\tableofcontents

\section{Introduction }
Throughout this paper $p\in \NN$ will denote a prime and all rings will be commutative with a unit. For a ring $B$ we use the notation $\Aut(\AA^1_{B}) = \Aut_{B}(B[T])^{op}$ and identify this group with collections of polynomials under composition. 

For $R$ be a $q$-torsion free ring where $q\in R$ and $qR\in \Spec(R)$ we will let $R_{d} = R/q^{d+1}$. 
In this paper we prove three theorems about univariate polynomial automorphisms over such rings $R$. 
Two important examples to keep in mind are $R_n = K[t]/(t^{n+1})$ for $K$ a field and $R_n = \ZZ/p^{n+1}$ which are the ``geometric'' and ``arithmetic'' cases in what follows.

\begin{thm}\label{thm: big}
 For $R$ a $q$-torsion free ring where $q\in R$ and $qR\in \Spec(R)$ the collection
 $$ \At_d(R,q) := \lbrace \psi \in \Aut( \AA^1_{R_{n-1}}): \forall m\geq 2, \ \ \deg(\psi \mod q^{m} ) \leq d 2^{m-2} \rbrace $$
 forms a subgroup under composition.
\end{thm}
This implies for example that for every $R$ a $q$-torsion free ring with $qR\in \Spec(R)$ and every $\psi \in \Aut(\AA^1_{R_n})$ that every iterate of $\psi$ has bounded degree.
In particular it implies that every $\psi \in \Aut(\AA^1_{\ZZ/p^{n+1}})$ has finite order.

We define the set of polynomials $A_d(R,q)$ which are of the form 
 \begin{equation}\label{eqn: form of polynomial}
 f(T) \equiv a_0 + a_1 T + q a_2 T^2 + q^2 a_3 T^4 + \cdots + q^{d-1} a_d T^d \in R_{d-1}[T]
 \end{equation}
which are invertible under composition. 
It was proved in \cite{Dupuy2013} that for any $q$-torsion free ring $R$ the set $A_{d}(R) \subset \Aut(\AA^1_{R_{n-1}})$ under composition is a group. 
We reproduce this proof in Example \ref{prop:A2_is_a_subgroup}. 

When $R=\ZZ$ and $q=p$ we have that the polynomial $f(T) = 1 + T + p T^2 + p^2 T^3 + p^3 T^4$ has finite order under composition mod $p^4$ and that every iterate has degree less or equal to four mod $p^4$.

\begin{thm}\label{thm:greenberg}
 \begin{itemize}
  \item There exists a finite dimensional group scheme over $\FF_p$ whose group of $\FF_p$-points is isomorphic to $A_n(\ZZ,p)$.
  \item There exists (an infinite dimemsional) group scheme over $\FF_p$ whose group of $\FF_p$-points is isomorphic to $\Aut(\AA^1_{\ZZ/p^n})$.
 \end{itemize}

\end{thm}
These groups can in some sense be considered as ``the Greenberg transform'' of $A_n(\ZZ,p)$ and $\Aut(\AA^1_{\ZZ/p^n})$ respectively.

 \begin{thm}
For $R$ a $q$-torsion free ring where $q\in R$ and $qR\in \Spec(R)$ the groups $A_n(R,q)$ and $\Aut(\AA^1_{R_n})$ are solvable. 
\end{thm}
In particular with implies that for every $m\in \ZZ$ the groups $\Aut(\AA^1_{\ZZ/m})$ are solvable.

\subsection{Motivation}
We will take a moment to motivate these groups. 
Let $B$ be a ring. 
Recall that an \textbf{$\AA^1$-bundle} over a scheme $X/B$ is a scheme $E/B$ together with a morphisms $\pi: E \to X$ with the property that for every point $x\in X$ there exists an affine open subset $U$ containing $x$ and an isomorphisms $\psi: \pi^{-1}(U) \cong U \times_B \AA^1_B$ with the property that $\pi\vert_{\pi^{-1}(U)} = p_1 \circ \psi$ where $p_1: U\times \AA^1 \to \AA^1$ is the first projection onto $U$. 
Given two such isomorphisms $\psi,\psi': \pi^{-1}(U) \cong U\times_B \AA^1_B$ we can consider the map
 $$ \psi' \circ \psi^{-1}:U\times_B \AA^1_B \to U \times_B\AA^1_B.$$
Maps of this form are essentailly the subject of this paper when $B$ is a non-reduced ring. These appear in the author's paper \cite{Dupuy2013} when $E$ is the first $p$-jet spaces of a curve modulo $p^n$.

\subsection{Plan of the paper}
In section \ref{groups} we introduce the groups and notation that we are going to use. 
In particular subsection \ref{sec:biggroup} proves theorem \ref{thm: big}.

In section \ref{greenberg} we prove theorem \ref{thm:greenberg} which shows that univariate polynomials automorphisms over $\ZZ/p^n$ are points of algebraic groups. This section starts by proving this in a simple example and moves to the more general case.

In section \ref{solvable} we prove solvability of univariate polynomial automorphisms by proving certain normal subgroups are abelian. This section starts by proving the theorem in a simple case then proves the more general case.

We apply the results developed in section \ref{solvable} in section \ref{adjoint}. In subsection \ref{inverses} we given an algorithm for computing inverse. 
The remainder of the subsection gives module structures to the abelian subgroups introduced in section \ref{solvable} and gives examples of explicit representations for the adjoint actions of on these abelian subgroups.

\subsection{Acknowledgements}
We would like to thank Alexandru Buium and Danny Krashen for their useful comments and suggestions.

\section{Automorphisms of the affine line over non-reduced rings}\label{groups}

We introduced the groups $\At_d(2,R,q)$ (denoted $\At_d$ in the author's thesis) of polynomial automorphisms modulo $p^2$ with degree bounded by $d$ and proved that they form a subgroup. This section aims to generalize the result groups $\At_d(2,R,q)$. 

\subsection{Subgroups of bounded degree}\label{sec:biggroup}

The main aim of this section is to prove the following result:
\begin{thm}\label{biggroup}
For $R$ a $q$-torsion free ring with $q\in R$ and $qR\in \Spec(R)$, the collection 
\begin{equation}
\At_d(n,R,q) = \lbrace \psi \in \Aut(\AA^1_{R_{n-1}}): \deg(\psi \mod p^m ) \leq 2^{m-2} d, \ \ \ 2\leq m \leq n\rbrace
\end{equation}
is a subgroup.
\end{thm}

Before proving this result in general we give several examples of the proof in special cases.

\begin{ex}\label{prop:A2_is_a_subgroup} 
We will show that for each $d\geq 1$, $\At_d(2,R,q)$ is a group. 

Let $\psi(T) = a_0 + a_1T+ q f(T)$ and $\psit(T) = \at_0 + \at_1T+q\ft(T)$ with $\ord_T(f),\ord_T(\ft)\geq 2$ we get
 \begin{eqnarray}
  \psi(\psit(T)) &\equiv& a_0 + a_1[ \at_0 + \at_1T + q \ft(T)] + qf( \at + \at T) \nonumber\\
  &=&
  a_0 + a_1\at_0 + (a_1\at_1)T \nonumber \\
   &&\ \ \ \ \ \ \ \ \ \ \ \ \ \ \ \ + q(a_1\ft(T) + f(\at_0+\at_1T)) \label{eqn:top_term_mod_p_squared} 
 \end{eqnarray}
Note that $a_1 \ft(T) + f(\at_0+\at_1T)$ has degree no larger than $\max\{ \deg f, \deg \ft \}$. 
Also, since $a_1$ and $\at_1$ are units, the degree of $\psi\circ \psit$ is exactly $\max\lbrace \deg \psi,\deg \psit \rbrace$ in the case that $\deg \psi \neq \deg \psit$. This means that if $\psi$ and $\psit$ are inverse to each other then $\deg \psi = \deg \psit$ and hence $\At_d(2,R,q)$ is closed under inverses. This is enough information to show that $\psi\circ \psit^{-1} \in \At_d(2,R,q)$ which shows $\At_d(2,R,q)$ is a subgroup. \qed
\end{ex}

\begin{ex}\label{mod p3}
We will show that the set of $\psi \in \Aut(\AA^1_{R_2}) $ satisfying
\begin{eqnarray*}
 \deg (\psi \mod q^2) \leq d \\
 \deg (\psi \mod q^3) \leq 2 d 
\end{eqnarray*}
form a subgroup. If we write
 \begin{eqnarray*}
  \psi(T) = a_0 + a_1 T + q f(T) + q^2 g(T) \mod q^3
 \end{eqnarray*}
for $f(T) \in R_1[T], g(T)\in R_0[T]$ with $\ord_T(f)\geq 2$ and $\ord_T(g)\geq 3$. 
Note that we have
\begin{eqnarray*}
\deg(\psi \mod q^2) &\geq& \deg( f\mod q),\\
\deg(\psi \mod q^3) &\geq& \deg( g \mod q) , \deg( f \mod q^2),
\end{eqnarray*}
Composing $\psi$ with $\psit$ gives
 \begin{eqnarray*}
  \psi(\psit(T)) &=& a_0 + a_1 \psit(T) \\
  && + q[ f(\at_0 + \at_1 T) + q f'(\at_0 + \at_1 T)\ft(T)  ]\\
  && + q^2g(\at_0 + \at_1 T)
 \end{eqnarray*}
Since the invertible polynomials of degree less that $d$ are a group modulo $q^2$ we only need to check that $\deg(\psi(\psit(T)))\leq 2d$.
We can just check each term is bounded by $2d$,
\begin{itemize}
 \item The degree of $f(\at_0+\at_1 T) \mod q^2 $ is bounded by $2d$,
 \item The degree of $f'(\at_0 + \at_1 T)\ft(T) \mod q$ is bounded by $(d-1)+d$,
 \item The degree of $g(\at_0+\at_1 T) \mod q$ is bounded by $2d$,
\end{itemize}
which completes the proof. \qed

\end{ex}

\begin{ex}
We will show that the set of $\psi \in \Aut(\AA^1_{R_3}) $ satisfying
\begin{eqnarray*}
 \deg (\psi \mod q^2) \leq d \\
 \deg (\psi \mod q^3) \leq 2 d \\
 \deg (\psi \mod q^4) \leq 4 d 
\end{eqnarray*}
form a subgroup. The proof relies on the previous two cases. 
We will write such a $\psi$ as 
 \begin{eqnarray*}
  \psi(T) = a_0 + a_1 T + q f(T) + q^2 g(T) + q^3 h(T) \mod q^4
 \end{eqnarray*}
for $f(T) \in R_2[T], g(T)\in R_1[T],h(T)\in R_0[T]$ where $\ord_T f \geq 2$, $\ord_T g\geq 3$ and $\ord_T h \geq 4$. 
Note that our conditions on degree imply that 
\begin{eqnarray*}
\deg(\psi \mod q^2) &\geq& \deg( f\mod q),\\
\deg(\psi \mod q^3) &\geq& \deg( g \mod q) , \deg( f \mod q^2), \\
\deg(\psi \mod q^4) &\geq& \deg( h \mod q), \deg( g \mod q^2), \deg(f \mod q^3).
\end{eqnarray*}
We will make use of these bounds in the subsequent computations. 
Composing $\psi$ with $\psit$ gives
 \begin{eqnarray*}
  \psi(\psit(T)) &=& a_0 + a_1 \psit(T) \\
  && + q[ f(\at_0 + \at_1 T) + q f'(\at_0 + \at_1 T)( \ft(T) + q \gt(T) ) + q^2 \frac{f''(\at_0+\at_1 T)}{2} \ft(T)^2]\\
  && + q^2[g(\at_0 + \at_1 T) + q g'(\at_0 + \at_1 T) \ft(T)  ]\\
  && + q^3 h(\at_0+\at_1T).
 \end{eqnarray*}
From example \ref{mod p3} we just need to show that each term in this polynomial is bounded by $4d$.
\begin{itemize}
 \item The degree of $f(\at_0+\at_1 T) \mod q^3$ is bounded by $4d$,
 \item The degree of $f'(\at_0 + \at_1 T) \ft(T) \mod q^2$ is bounded by $(2d-1)+(2d)$
 \item The degree of $f'(\at_0 + \at_1 T) \gt(T) \mod q$ is bounded by $(d-1)+(2d)$
 \item The degree of $f''(\at_0+\at_1 T)\ft(T)^2 \mod q$ is bounded by $(d-2)+ 2d$
 \item The degree of $g(\at_0 + \at_1 T) \mod q^2 $ is bounded by $4d$,
 \item The degree of $g'(\at_0 + \at_1 T) \ft(T) \mod q $ is bounded by $(2d-1)+d$,
 \item The degree of $h(\at_0+\at_1T)$ is bounded by $4d$.
\end{itemize}
This shows that $\psi(\psit(T))$ has degree less than $4d$.  \qed
\end{ex}

We will now give a theorem which generalizes the examples above.

\begin{proof}[Proof of Theorem \ref{biggroup}]
 We prove this by induction on $n$ where the statement is that the set $\At_d(n,R,q)$ is a group. 
 The base case is proved since $\At_d(2,R,q)$ is a subgroup.
 We will suppose that $\At_d(n-1,R,q)$ is a subgroup. 
 Let $\psi(T),\psit(T) \in \At_d(n,R,q)$. 
 We just need to show that $\psi \circ \psit\in \At_d(n,R,q)$, i.e. that the degree of $\psi\circ \psit$ is bounded by $2^{n-2}d$.
 
 We will write $\psi$ as
 $$ \psi(T) = a_0 + a_1 T + qf_1(T) + q^2 f_2(T) + \cdots + q^{n-1} f_{n-1}(T) \mod q^n $$
 where $f_i$ is defined $q^{n-i}$ and $\ord_T f_i \geq i+1$.
 Note that we have 
 $$ d2^{m-2}\geq \deg( \psi(T) \mod q^m ) \geq \deg( f_i(T) \mod q^{m-i} ) $$ 
For $2 \leq m \leq n-1$ and $2 \leq i\leq m-1$.
We will now examine each of the terms of $\psi \circ \psit$. 
Here we have terms 
\begin{eqnarray*}
 \psi(\psit) &=& f_0(\psi) + q f_1(\psi) + \cdots + q^{n-1}f_{n-1}(\psit) \mod q^n.
\end{eqnarray*}
The general term is 
\begin{eqnarray}
 q^i f_i(\psit) &=& q^i f_i(\ft_0 + q \ft_1 + q^2 \ft_2 + \cdots + q^{n-i-1} \ft_{n-i-1}) \mod q^n \nonumber \\
 &=& q^i[ f_i(\ft_0) + f_i'(\ft_0) A + \cdots + \frac{f_i^{(n-i-1)}(\ft_0)}{(n-i-1)!}A^{n-i-1}] \label{terms}
\end{eqnarray}
where
$$ A = \ft_1 + q \ft_2 + q^2 \ft_3 + \cdots + q^{n-i-2}\ft_{n-i-1}. $$
\begin{lem}\label{estimates}
 For $i=1,\ldots,n$ the $j$th term in \ref{terms} is $q^{i+j}\frac{f_i^{(j)}(\ft_0)}{j!}A^j$ and $0\leq j \leq n-i-1$. 
 The degree of this term is bounded by 
  $$ d_{n-j-1} - j + jd_{n-i-j} $$
  where $\deg(\psi \mod q^n),\deg(\psit \mod q^n) \leq d_{n-1}$ for each $n\leq m$. (In our application $d_j=2^{j-2}d$.)
\end{lem}
\begin{proof}
First, 
  $$q^{i+j} A = q^{i+j}( \ft_1 + q \ft_2 + \cdots + q^{n-i-2} \ft_{n-i-1}$$
Since this is a term in $q^{i+j-1} \psit \mod q^n$ and corresponds to a term in $\psit \mod q^{n-i-j+1}$ which has degree bounded by $d_{n-i-j}$. 
Since $A$ appears in the with multiplicity $j$ in $q^{i+j}\frac{f_i^{(j)}(\ft_0)}{j!}A^j$ it contributes $j\cdot d_{n-i-j}$ to the degree bound of $q^{i+j}\frac{f_i^{(j)}(\ft_0)}{j!}A^j$.

Now we look at $q^{i+j}f_i^{(j)}(\ft_0)$. The expression $q^{i+j}f_i$ is a term in $q^j\psi \mod q^n$ which corresponds to a term of $\psi \mod q^{n-j}$ and hence has a degree bounded by $d_{n-j-1}$. Since we are taking $j$ derivatives in our expression it has a contribution of $d_{n-j-1}-j$. 

Putting the information from the two factors of $q^{i+j}\frac{f_i^{(j)}(\ft_0)}{j!}A^j$ together we have an overall degree bound 
$$ d_{n-j-1} -j + j d_{n-i-j}$$
as advertised.
\end{proof}
All the terms in expression \ref{terms} are of the form $q^{i+j}\frac{f_i^{(j)}(\ft_0)}{j!}A^j$ where $i$ varies from $0$ to $n-1$ and $j$ varies from $0$ to $n-i-1$. 
We will now use the estimates in Lemma \ref{estimates} to finish our proof. 
Again, we suppose that $d_j = 2^{j-1}d$. Plugging this into our expression we have
\begin{eqnarray*}
 2^{n-j-2}d-j + j 2^{n-i-j-1}d &\leq& 2^{n-j-2}d + j 2^{n-i-j-1}d  \\
 &\leq& 2^{n-j-2}d + j 2^{n-j-1}d \\
 &=& 2^{n-2}d \left( 2^{-j} + j 2^{-j-2} \right) 
\end{eqnarray*}
and since $2^{-j} + j 2^{-j-2}\leq 1$ for all $j\geq 0$ we have our desired bound. 
\end{proof}

\begin{ex}
 The polynomial 
 $$\psi(T) = T + q T^d + q^2 T^{2d} + q^3 T^{4d} + q^4 T^{8d} + q^5 T^{16d} \mod q^6 \in \At_d(6,\ZZ,p).$$
 In particular has finite order under composition. 
\end{ex}

\begin{cor}
 Let $\psi \in \Aut(\AA^1_{\ZZ/q^n})$ and let $\psi^r(T) = \psi(\psi(\cdots (\psi(T))\cdot ) )$ where the composition occurs $r$ times. For every $r\geq 1$ we have 
   $$ \deg(\psi^r(T)) \leq 2^{n-2}( \deg(\psi \mod q^2) ).$$
\end{cor}
\begin{proof}
 Take $\psi$ of degree $d$. 
 It certainly has degree $d \mod q^2$ satisfies $\deg(\psi \mod q^n) \leq d 2^{n-2}$ for each $n\geq 2$ so $\psi$ is in the subgroup $\At(d,R,q)$.
\end{proof}

\section{The ``Greenburg transform'' of univariate polynomial automorphisms}\label{greenberg}
Let $R$ be a $q$-torsion free ring where $q\in R$ and $qR \in \Spec(R)$. 
In \cite{Dupuy2013} (section 4.1) we introduced the groups $A_d(R,q) \subset \Aut(\AA^1_{R_{d-1}})$ consisting of polynomial automorphisms of the form 
$$ \psi(T) = a_0 + a_1 T + q a_1 T^2 + \cdots + q^{d-1} a_{d-1} T^d \ \ \ \mod q^d$$ 
and proved that they were a subgroup.

The aim of this section is to prove the following theorem:
\begin{thm}
 The groups $A_n(\ZZ,p)$ and $\Aut(\AA^1_{\ZZ/p^n})$ are isomorphic to group of $\FF_p$-points of an algebraic group.
\end{thm}
This follows from a Greenberg-like transform. 

\subsection{Witt vectors}
An excellent reference for Witt vectors is chapter one of \cite{Hazewinkel2009}. Let $R$ be a $p$-torsion free ring. The ring of \textbf{$p$-typical Witt vectors} of $R$, $W(R)$ is the set $R^{\NN}$ together with a \textbf{Witt addition} and \textbf{Witt multiplication} which define a ring structure:
\begin{eqnarray*} 
 [x_0,x_1,x_2,\ldots] +_W [y_0,y_1,y_2,\ldots] &=& [s_0,s_1,s_2,\ldots],\\
 {} [x_0,x_1,x_2,\ldots] *_W [y_0,y_1,y_2,\ldots] &=& [m_0,m_1,m_2,\ldots].
\end{eqnarray*}
Here $s_i,m_i \in \ZZ[x_0,x_1,\ldots,x_i,y_0,y_1,\ldots,y_i]$ the \textbf{Witt addition} and \textbf{Witt multiplication} polynomials. They are the unique polynomials so that for every ring $A$ the \textbf{Ghost Map}
\begin{eqnarray*}
 w:`W(A) &\to& A^{\NN}\\
{} [x_0,x_1,x_2,\ldots] &\mapsto& [w_0,w_1,w_2,\ldots]
\end{eqnarray*}
is a ring homomorphism. Here $w_j(x) = \sum_{n=0}^j p^j x_j^{p^{n-j}}$ are the \textbf{Witt polynomials}. Also, in the map $w$ above, we give $A^{\NN}$ its usual componentwise addition and multiplication. 

\begin{ex}
 To compute the first two Witt addition and multiplication polynomials one needs to solve the ``universal'' system of equations\footnote{ Universal meaning that we view everything just as symbols and don't worry about what ring or characteristic they are in.}
 \begin{eqnarray*}
  w_0(x+_W y) &=& w_0(x) + w_0(y)\\
  w_1(x+_W y) &=& w_1(x) + w_1(y)\\
    w_0(x*_W y) &=& w_0(x) w_0(y)\\
  w_1(x*_W y) &=& w_1(x)w_1(y)
 \end{eqnarray*}
 which amounts to the system
 \begin{eqnarray*}
  s_0 &=& x_0 + y_0\\
  s_0^p + p s_1 &=& (x_0^p + p x_1) + (y_0^p + p y_1)\\
  m_0 &=& x_0y_0\\
  m_0^p + p m_1 &=& (x_0^p +p x_1)(y_0^p + p y_1)
 \end{eqnarray*}
which has the solution
\begin{eqnarray*}
 s_0 &=& x_0 + y_0\\
 s_1 &=& x_1 + x_1 - \sum_{j=1}^{p-1} \frac{1}{p} {p \choose j} x^j y^{p-j}\\
 m_0 &=& x_0y_0\\
 m_1 &=& x_1 y_0^p + y_0^p y_1 + p x_1 y_1
\end{eqnarray*}
 \qed

\end{ex}
It is a theorem of Witt's that these can be solved to give integral sum and addition polynomials. 

If we make a ring out of the set $R^{n}$ by using only the first $n$ Witt addition and multiplication polynomials we get the \textbf{ring of truncated $p$-typical Witt vectors} $W_{n-1}(R)$. We will also make use of the following important property of Witt vectors
\begin{thm}[Witt]\label{thm: witt and mod}
   $W_n(\FF_p) \cong \ZZ/p^{n+1}$
\end{thm}

\subsection{The Greenberg transform}
The Greenberg transform first appeared in Lang's thesis and can be found in \cite{Lang1952} and \cite{Lang1954}. It is essentially a way of converting polynomials over $\ZZ/p^n$ to polynomials in more indeterminates over $\FF_p$ using Witt vectors.

\begin{ex}
 To compute the second Greenberg transform of $f(x,y)=x^2 +y \in \ZZ[x,y]$ we compute $x^2+y$ using Witt additions and Witt multiplications
 \begin{eqnarray*}
 f([x_0,x_1],[y_0,y_1]) &=& [x_0,x_1]^2 + [y_0,y_1] \\
 &=& [x_0^2, 2x_0^px_1 + p x_1^2] +[y_0,y_1] \\
 &=& [x_0^2y_0, 2x_0^p x_1 + p x_1^2 + y_1 - \sum_{j=1}^{p-1} \frac{1}{p} { p\choose j} x_0^{2j} y_0^{p-j}] 
 \end{eqnarray*}
and get the polynomals
\begin{eqnarray*}
 x_0^2y_0, & 2x_0^p x_1 + p x_1^2 + y_1 - \sum_{j=1}^{p-1} \frac{1}{p} { p\choose j} x_0^{2j} y_0^{p-j}.
\end{eqnarray*}
\qed
\end{ex}

In general, the \textbf{$n$th Greenberg Transform of a polynomial} $f(x_0,x_1,\ldots,x_m) \in \ZZ[x_0,x_1,\ldots,x_m]$ is the set of polynomials
 $$g_0,g_1,\ldots,g_n \in \ZZ[x_{0,1},\ldots,x_{0,n}; x_{1,0}, \ldots, x_{1,n}; \ldots; x_{m,0},\ldots,x_{m,n} ]$$
where the $g_i$ are defined by the equation
$$[g_0,\ldots,g_n] = f([x_{0,0},\ldots,x_{0,n}],\ldots,[x_{n,0},\ldots,x_{n,n}]).$$

The \textbf{$n$th Greenberg transform of an ideal} $I \idealin \ZZ[x_0,x_1,\ldots,x_m]$ is the ideal $I'\idealin \ZZ[x_{0,1},\ldots,x_{0,n}; x_{1,0}, \ldots, x_{1,n}; \ldots; x_{m,0},\ldots,x_{m,n} ]$ generated by the $n$th Greenberg transforms of the polynomials in $I$. The \textbf{$n$th Greenberg transform of a subscheme of affine space} $X=V(I) \subset \AA^{m}_{\ZZ}$ is the scheme $ \Gr_n X \subset \AA^{nm}_{\ZZ}$ defined by $\Gr_nX = V(I')$ where $I'$ is the $n$th Greenberg transform of $I$. 

In general on can define the $n$th Greenberg transform as a functor $ \Gr_n: \Sch_{\ZZ} \to \Sch_{\FF_p}. $
This functor has the property that $(\Gr_nX)(\FF_p) = X(\ZZ/p^n).$

\subsection{Automorphisms of the affine line as points of algebraic groups}

We begin with an instructive example. 

\begin{ex}\label{greenberg of a2}
We will explain how to apply the Greenberg Transform to $A_2(\ZZ,p)$ of univariate polynomial automorphisms of degree two modulo $p^2$. We identify $f(T) = a + bT + pc T^2$ and $g(T) = a'+b'T + pc'T^2$ with $[a_0,a_1] + [b_0,b_1] T + [0,c_1]T^2$ and $[a_0',a_1']+[b_0',b_1']T + [0,b_1]T^2$ then we multiply out the vectors as we normally would giving Witt multiplications
\begin{eqnarray*}
 && [a_0,a_1] + [b_0,b_1] T + [0,c_1]T^2 \circ  [a_0',a_1'] + [b_0',b_1'] T + [0,c_1']T^2 \\
 &=& [a_0,a_1] + [b_0,b_1] [a_0',a_1'] + [0,c_1][a_0',a_1']^2 \\
 && +([b_0,b_1][b_0',b_1'] + 2[0,c_1][b_0',b_1'][a_0',a_1'])T \\
 && +([b_0,b_1][0,c_1'] + [0,c_1][b_0',b_1']^2)
\end{eqnarray*}
which gives
\begin{eqnarray*} 
 \ [a_0'',a_1''] &=& [a_0,a_1] + [b_0,b_1] [a_0',a_1'] + [0,c_1][a_0',a_1']^2 \\     
\  [b_0',b_1''] &=& [b_0,b_1][b_0',b_1'] + 2[0,c_1][b_0',b_1'][a_0',a_1'] \\
\  [c_0'',c_1''] &=& [b_0,b_1][0,c_1'] + [0,c_1'][b_0,b_1']^2
\end{eqnarray*}
which tells us how to transform the coordinates. One can multiply these out to get explicit polynomials using the rules for Witt addition and Witt multiplication:
\begin{eqnarray*}
 a_0'' &=& a_0 + b_0a_0' \\
 a_1'' &=& a_1 + b_1(a_0')^p + p a_1' b_1 + c_1(a_0')^{2p} + pc_1( 2 a_1'(a_0')^p + p (a_1')^2 ) \\
 b_0'' &=& b_0b_0' \\
 b_1'' &=& b_0^pb_1' + (b_0')^pb_1+pb_1b_1' + 2 ( (b_0'a_0')^p c_1 + p c_1( (b_0')^p a_1' + b_1'(b_0')^p + p b_1'a_1'))\\
 c_0'' &=& 0 \\
 c_1'' &=& b_0^p c_1' + p b_1 c_1' + (b_0')^{2p}c_1 + p c_1( 2 (b_0')^p b_1' + p (b_1')^2 
\end{eqnarray*}
If we assume that $a_0,a_1,b_0,b_1,c_1,a_0',a_1',b_0',b_1',c_1'$ are in $\FF_p$ these simplify to 
\begin{eqnarray*}
 a_0'' &=& a_0 + b_0a_0' \\
 a_1'' &=& a_1 + b_1a_0' c_1(a_0')^{2}\\
 b_0'' &=& b_0b_0' \\
 b_1'' &=& b_0 b_1' + b_0'b_1+ 2 ( b_0'a_0' c_1 + b_1'1b_0' )\\
 c_0'' &=& 0 \\
 c_1'' &=& b_0 c_1' +(b_0')^{2}c_1.
\end{eqnarray*}
These relations define an algebraic group $G/\FF_p$ where $G\cong \AA^5_{\FF_p}$ as varieties and the group multiplication is given by
$$\mu( (a_0,a_1,b_0,b_1,c_1), (a_0',a_1',b_0',b_1',c_1') ) = (a_0'',a_1'',b_0'',b_1'',c_1'').$$
From Theorem \ref{thm: witt and mod} it is clear that $G(\FF_p) \cong A_2(\ZZ)$.
  \qed
\end{ex}

The following theorem will define what we mean by ``The Greenberg Transform'' of the group $A_d(\ZZ,p)$. 
\begin{thm}
 There exists an algebraic group $G/\FF_p$ which is a subscheme of $\AA^{N}$ where $N = n+\frac{n(n+1)}{2}$ with the the property that 
 $$ G(\FF_p) \cong A_d(\ZZ,p).$$
\end{thm}
\begin{proof}
We proceed as in example \ref{greenberg of a2} and identify
$$ a_0 + a_1 T + p a_2 T^2 + p^2 a_3 T^3 + \cdots + p^{n-1} a_n T^n \mod p^n $$
with 
$$ [a_{0,0},\ldots, a_{0,n-1}] + [a_{1,0},\ldots, a_{1,n-1}] T + [0,a_{2,1},\ldots,a_{2,n-1}] T^2 + \cdots + [0,0,\ldots,0,a_{n,n-1}] T^n.$$
We then multiply out two such vectors using Witt addition and Witt multiplication to get some algebraic relations.  

The formula for $N$ comes from the adding $n+n+(n-1) + \cdots + 1$.
\end{proof}

\begin{rem}
 ``The Greenberg Transform'' of $A_d(\ZZ,p)$ isn't a genuine Greenberg Transform since $A_d(\ZZ,p)$ isn't a scheme.
\end{rem}

The proof for the full group of automorphisms is quite similar. 
\begin{thm}
 The group $\Aut(\AA^1_{\ZZ/p^{n+1}})$ is isomorphic to the $\FF_p$ points of an algebraic variety.
\end{thm}
\begin{proof}
 The proof is similar. We replace coordinate $a_0,a_1,\ldots$ appearing in polynomials
 $$ a_0 + a_1 T + a_2 T^2 + \cdots + a_d T^d \mod p^{n+1} \in \Aut(\AA^1_{\ZZ/p^{n+1}}) $$
 with Witt coordinates $[a_{i0},\ldots, a_{in}]$ where $0\leq i $ and impose the additions and multiplications as usual. The only difference from the proof for $A_n(\ZZ,n,p)$ is that that we have an infinite number of indeterminates and that we need to stipulate that the polynomial are affine linear modulo $p$. 
 To do this we adjoint an extra symbol $y$ and the additional equation $a_{10}y-1$ as usual in algebraic geometry. 
\end{proof}

\section{Solvablity}\label{solvable}

\subsection{Solvable groups}
Recall that a group $G$ is \textbf{solvable} if and only if it admits a composition series 
	$1=G_0 \idealin G_1 \idealin \ldots \idealin G_l = G$
such that for every $0\leq i\leq l$ the factor groups
	$G_{i+1}/G_{i} =: A_{i+1}$
are abelian. 

We construct the class of groups which are \textbf{built by abelians} inductively: A group $G$ is built by abelians if one of the follow holds
\begin{description}
 \item[base case] $G$ is abelian
 \item[inductive step] 
 \begin{itemize}
  \item $G$ is an extension of a group built by abelians by an abelian group.
  \item $G$ is an extension of an abelian group by a group built by abelians.
 \end{itemize}
\end{description}

\begin{lem}
A group $G$ is solvable if and only if it is built by abelians.
\end{lem}
Since we don't know of a good reference we give the proof here.
\begin{proof}
Suppose that $G$ is solvable and let $l(G)$ denote the length of the minimal composition series. 
We will show that it is built by abelians by induction on the length $l(G)$ of a minimal composition series for $G$.

Suppose that $1(G)=1$. Then $G_1=G_1/G_0 = A_1$ and $G_1$ is abelian.

Now suppose the proposition if true for $l(G)=n-1$. Given a composition series for $G=G_n$ we have a composition series for $G_{n-1}$ which shows that $G_{n-1}$ is solvable. By inductive hypothesis $G_{n-1}$ is built by abelians. The exact sequence
	$$1 \to G_{n-1} \to G_n  \to A_n \to 1,$$
which shows that $G_n$ is built by abelians. 
In particular every solvable group is built from extending abelian groups by solvable groups.

We will now prove the converse. 
Let $G$ be a group built by abelians. 
Let $c(G)$ be the minimal number of admissible extensions required to built $G$. 
Our proof will be by induction on $c(G)$. 
If $c(G)=1$ then $G$ is solvable since every abelian group is solvable.

Suppose $c(G)=n$ and that $G=G_n$ is an extension of an abelian group $A_n$ by a group built by abelians $G_{n-1}$. This means we have an exact sequence  
	$$ 1 \to G_{n-1} \to G_{n} \to  A_{n} \to 1$$.
which implies that $G_n$ is solvable since $G_{n-1}$ is by inductive hypothesis.

Suppose now that $G_n$ is an extension of $G_{n-1}$ by an abelian group $A$: 
$$1 \to A \to G_{n} \to G_{n-1}. $$
Let $p_n: G_n \to G_{n-1}$ with $\ker(p_n) = A$ as above. 
Since $G_{n-1}$ is built from abelians it is solvable by inductive hypothesis. 
In particular there exists a sequence of groups subgroups 
$$1=\Gamma_0 < \ldots < \Gamma_m =  G_{n-1}$$
such that $\Gamma_j/\Gamma_{j-1} = B_j$ where $A$ is abelian.
Define $G_j' = \pi^{-1}(\Gamma_j)$. 
We have $G_{j-1}' \idealin G_{j}'$ and $A \subset G_j$.
We also have $G_j'/A \cong \Gamma_j$ so $G_j'/G_{j-1}' \cong (G_j'/A)/(G_{j-1}'/A) \cong \Gamma_j/\Gamma_{j-1} = B_j$ which is abelian.
In addition $G_0'=A$ so we have constructed a composition series for $G_n$ and hence $G_n$ is abelian.

\end{proof}
This lemma just says that solvable groups are built from abelian groups. 

\subsection{Abelian normal subgroups}\label{abelian kernels}

The following lemma will allow us to build the groups $A_d(R,q)$ out of abelian ones.
\begin{lem}\label{prop: kernel}
 Let $R$ be a $q$-torsion free ring with $qR\in \Spec(R)$. The kernel the natural map $\pi_d:A_d(R,q) \to A_{d-1}(R,q)$ is isomorphic to $R_0^{d+1}$.
\end{lem}
\begin{proof} The group $N_d(R,q):= \ker(\pi_d:A_d(R,q) \to A_{d-1}(R,q))$ consists of elements of the form
 \begin{equation}
  \psi(T) = q^{d-1} a_0 + (1 + a_1 q^{d-1}) T + q^{d-1} a_2 T^2 + \cdots + q^{d-1} a_d T^d \mod q^d
 \end{equation}
whose reduction mod $q^{d-1}$ is the identity. This is clearly a normal subgroup. We now show that it is closed under composition; every $\psi$ can be written as $\psi(T) = T + q^{d-1}\psi'(T)$. Suppose that $\varphi(T) = T + q^{d-1}\varphi'(T)$ then we have
$$ \varphi(\psi(T)) = \psi(T) + q^{d-1}\varphi'(T + q^{d-1} \psi'(T)) = \psi(T) + \varphi'(T) = T + \psi'(T) + \varphi'(T) \mod q^{d}.$$
From this expression it is now clear that $N_d(R,q) \cong R_0^d$ where the isomorphism is given by
 $$ q^{d-1} a_0 + (1 + a_1 q^{d-1}) T + q^{d-1} a_2 T^2 + \cdots + q^{d-1} a_d T^d \mod q^d \mapsto (a_0,a_1,a_2,\cdots,a_d).$$
\end{proof}

\begin{thm}
If $R$ is a $q$-torsion free ring with $qR\in \Spec(R)$ then the groups $A_d(R,q)$ are solvable.
\end{thm}
\begin{proof}
The proof is by induction on $d$. For $d=1$, $A_1(R,q) \cong R_0 \otimes R_0^{\times}$ which is clearly solvable. We will now assume the proposition is true for $d-1$ and prove it for $d$. For each $d$ we have the exact sequence
$$N_d(R,q) \to A_d(R,q) \to A_{d-1}(R,q) \to  1 $$
Which shows that $A_d(R,q)$ is an extension of $A_{d-1}(R,q)$ by the abelian group $N_d(R,q) \cong R_0^{d+1}$ and since $A_{d-1}(R,q)$ is solvable by hypothesis we are done.
\end{proof}

We spend the rest of this section generalizing the above results for the groups $\Aut(\AA^1_{R/q^n})$ and $\Aut(\AA^1_{\ZZ/m})$. We will show that both of these groups are solvable. 

We start with the following simple lemma:
\begin{lem}\label{lem: commuting}
Let $R$ be a $q$-torsion free ring and suppose $f_1,f_2 \in \Aut(\AA^1_{R/q^n})$. 
If $r+s\geq n$ and
 \begin{eqnarray*}
  f_1(T) &=& T \mod q^r\\
  f_2(T) &=& T \mod q^s 
 \end{eqnarray*}
 then
 \begin{equation}
  f_1 \circ f_2 = f_2 \circ f_1 \mod q^n
 \end{equation}
\end{lem}
\begin{proof}
 Write $f_1(T) = T + q^r g_1(T)$ and $f_2(T) = T + q^s g_2(T)$ then we have 
 \begin{eqnarray*}
  f_1(f_2(T)) &=& f_2(T) + q^r g_1(f_2(T)) \\
  &=& T + q^s g_2(T) + q^rg_1(T + q^s g_2(T) ) \\
  &=& T + q^s g_2(T) + q^r g_1(T) \mod q^n 
 \end{eqnarray*}
where the last line follows from the fact that if $a \equiv b \mod q^r$ then $a q^s \equiv b q^s \mod q^n$.
\end{proof}

\begin{cor}
Let $R$ be a $q$-torsion free ring. Let
 \begin{eqnarray*}
 N_{n,r}(R,q) &:=& \lbrace f \in A_n(R) : f(T) \equiv T \mod q^r \rbrace \\
&=& \ker(\pi_{n.r}: A_n \to A_r ),\\
K_{n,r}(R,q) &:=& \lbrace f \in \Aut(\AA^1_{R_{n-1}}) : f(T) \equiv T \mod q^r \rbrace \\
&=& \ker(\pi_{n.r}: \Aut(\AA^1_{R_{n-1}}) \to \Aut(\AA^1_{R_{r-1}})),
 \end{eqnarray*}
 If $r>n/2$ then both of these groups are abelian.
\end{cor}
\begin{proof}
 This follows from Lemma \ref{lem: commuting} since every pair of polynomial in $K_{n,r}(R,q)$ for $r>n/2$ commutes. Since $N_{n,r}(R,q) \subset K_{n,r}(R,q)$ we are done.
\end{proof}

The arguement in the proof of Lemma \ref{lem: commuting} can actually be used to prove something slightly more general.
\begin{lem}
 Let $R$ be a ring and $I,J\idealin R$ with $I^2=0$. The group 
  $$ \ker( \Aut(\AA^1_{R}) \to \Aut(\AA^1_{R/I}) )$$
 is abelian.
\end{lem}

This means for $R = \ZZ/m$ where $m = p_1^{n_1}\cdots p_s^{n_s}$ and $I = (m') \subset \ZZ/m$ where $m' = p_1^{r_1}\cdots p_s^{r_s}$ and $r_i>n_i/2$ for $i=1,\ldots,s$ we can apply our technique of solvability. 
We summarize our discussion in the following theorem.
\begin{thm}
 The following groups are solvable
 \begin{itemize}
  \item $\Aut(\AA^1_{R/q^n})$, where $R$ $q$-torsion free
  \item $\Aut(\AA^1_{\ZZ/m})$
 \end{itemize}
\end{thm}
\begin{proof}
 These groups are built by Abelians.
\end{proof}

\begin{rem}
 The author recognizes that he could have simply proved that $\Aut(\AA^1_{R/q^n})$ was solvable first and then used the fact that $A_d(R,q)$ was a solvable group to prove solvability here but decided to present it this way as this was the way he proved it first.
\end{rem}

\section{``Adjoint representations''} \label{adjoint}

\subsection{An algorithm for computing inverses}\label{inverses}
Let $R$ be a $q$-torsion free ring where $q\in R$ and $qR \in \Spec(R)$. 
We now move to the question of computing inverses in the group $\Aut(\AA^1_{R/q^n})$ efficiently.
Note that if $\psi(T) = T + q^r f(T) \in K_{n,r}(R,q)$ then its inverse is easily computable since the group is abelian and isomorphic $R/q^{n-r}[T]$.

Also note that if $\psi(T) = a_0 + a_1 T \in A_1(R,1) = \Aut(\AA^1_{R/q})$ its inverse is also easily computable. 
Suppose that $\psi(T) \in \Aut(\AA^1_{R/q^n})$ and let $\phi(T)$ be a lift of the inverse of $\pi_{n,r}(\psi)$ where $r>d/2$. Then $\psi \circ \phi \in K_{n,r}(R,q)$ and its inverse is readily computable. This gives a recurrsive algorithm for computing inverses.

This gives us the following recurrsive algorithm for computing inverses
\begin{alg} For $\psi \in \Aut(\AA^1_{R_{n-1}})$ we can compute $\psi^{-1}$ using
\begin{equation}
 \psi^{-1}(T) = 
\begin{cases}
  T-q^rf(T)&, \psi \in K_{n,r}(R,q), r>n/2\\
 -a_1^{-1}a_0 + a_1^{-1}T, & \psi \in A_1(R,q), \\
 \lift(\pi_{n,\lceil n/2 \rceil}(\psi)^{-1}) \circ (\psi \circ \lift( \pi_{n,\lceil n/2\rceil}(\psi)^{-1}))^{-1}, & \psi\in \Aut(\AA^1_{R_{n-1}}) \setminus N_{n,\lceil n/2 \rceil}  \mbox{ and } d\neq 1
\end{cases}
\end{equation}
Where $\lift:\Aut(\AA^1_{R_{n-1}})\to \Aut(\AA^1_{R_{d-1}})$ is just a map of sets such that $\pi_{n,r} \circ \lift = \id$. 
\end{alg}

\subsection{The adjoint representation}
Let $R$ be a $q$-torsion free ring with $qR\in \Spec(R)$. In the previous section we defined the groups
 $$ K_{n,r}(R,q):= \ker( \pi_{n,r}: \Aut(\AA^1_{R_{n-1}}) \to \Aut(\AA^1_{R_{r-1}}))$$
which had the property that they were Abelian when $r>n/2$.
There are the elements of $\Aut(\AA^1_{R_{n-1}})$ which can be thought of as $q$-adically close to the identity and hence should be viewed as the ``Lie algebra'' of $\Aut(\AA^1_{R_{n-1}})$. 
It is natural then to ask if $K_{n,r}(R,q)$ is an $R$-module and if there exists an $R$-linear adjoint action.
The answer to both these questions is yes which we will now show.

In what follows we define the \textbf{Adjoint Action} of $\Aut( \AA^1_{R_{n-1}})$ on $K_{n,r}(R,q)$ by
\begin{equation}
 \Ad_f(g) = f\circ g \circ f^{-1}
\end{equation}
for $f\in \Aut(\AA^1_{R_{n-1}})$ and $g\in K_{n,r}(R,q)$.

Define the $R$-multiplication on $K_{n,r}(R)$ by
\begin{equation}
c\dot g(T) = c \cdot ( T + q^r h(T) ) := T + q^r c h(T).
\end{equation}
Where $g(T) = T + q^r h(T) \in N_{n,r}(R,q)$ and $c\in \OO$.
We have the following theorem
\begin{thm}
 The adjoint action of $\Aut(\AA^1_{R_{n-1}})$ on $K_{n,r}(R)$ is $R$-linear for $r>n/2$. 
 That is, for all $f \in \Aut(\AA^1_{R_{n-1}})$, all $g \in K_{n,r}(R)$ and all $c \in R$ we have
 \begin{equation}
  \Ad_f(c\cdot g) = c\cdot \Ad_f(g).
 \end{equation}
\end{thm}
\begin{proof}
 Take $f \in \Aut(\AA^1_{R_{n-1}})$ and $g\in K_{n,r}(R,q)$ and write it as $g(T) = T + q^s h(T)$.
 \begin{eqnarray*}
  f \circ (c \cdot g) \circ f^{-1}(T) &=& f( T + c q^s  h(T) ) \circ f^{-1}(T) \\
  &=& f( f^{-1}(T) + c q^s h(f^{-1}(T)) ) \\
  &=& f( f^{-1}(T) + c q^s  h(f^{-1}(T)) ) \\
  &=& a_0 + a_1 (f^{-1}(T) + cq^s  h(f^{-1}(T)) \\
  && + \sum_{j=1}^{n-1} q^{j-1} a_j (f^{-1}(T) + c q^s  h(f^{-1}(T))^j \\
  &=& a_0 + a_1 (f^{-1}(T) ) + \sum_{j=1}^{n} q^{j-1} a_j f^{-1}(T)^j + a_1 c q^s h(f^{-1}(T)) \\
  &&  + \sum_{j=1}^{n} q^{j-1} a_j \left [ \sum_{l=1}^j { j \choose l } f^{-1}(T)^{j-l} ( c q^s  h(f^{-1}(T)) )^l \right ] \\
  &=& T + a_1 c q^s h(f^{-1}(T)) + \sum_{j=1}^{n} q^{j-1} a_j \left [ j f^{-1}(T)^{j-1} ( c q^s  h(f^{-1}(T)) ) \right ] \\
  &=& T + c q^s h(f^{-1}(T)) \left ( a_1  + \sum_{j=1}^{n} q^{j-1} j a_j f^{-1}(T)^{j-1} \right ) \\
  &=& T + c q^s  h(f^{-1}(T)) f'( f^{-1}(T) ) 
 \end{eqnarray*}
Where we reduced the sum in the binomial expansion using the fact that $(q^s)^l = 0 \mod q^n$ for $l>1$.

If we can show that 
\begin{equation}
 f \circ g \circ f^{-1}(T) = T +  q^s  h(f^{-1}(T)) f'( f^{-1}(T) ) 
\end{equation}
we are done. This is indeed the case if we take the above computation with $c=1$.
\end{proof}

This result is surprising since it says that we can study composition of polynomials over $\ZZ/p^n$ using representations. 

\begin{cor}
We have 
\begin{equation}
 N_{2m,m}(R,q) \cong ( R/q^m )^{\oplus m+1} \oplus R/q^m \oplus R/q^{m-1} \oplus \cdots R/q
\end{equation}
as $R$-modules.
\end{cor}

\subsection{Examples of explicit representations}
In this section we work over the ``universal ring''
 $$ R = \ZZ[a,b,c,d,1/b][q]. $$

\begin{ex}
 The group $N_{4,2}(R,q)$ consist of a subgroup of degree four polynomials mod $q^4$ and we have $N_{4,2}(R,q) \cong (R/q^2)^{\oplus 4} \oplus R/q$.
\end{ex}

The group action of $A_d(R,q)$ on the group $N_d(R,q)$ by conjugation gives a linear map
$$\Ad: A_d(R,q) \to \GL_{d+1}(R_0)$$.

The kernel of this map contains $N_d(R,q)$ which means that $\Ad_{d-1}(R,q)$ is well defined on the quotient $A_d(R,q)/N_d(R,q) \cong A_{d-1}(R,q)$. 

We can compute several representations of $A_n$ acting on certain subgroups
\begin{ex}
The action of $a + b + q c T^2 + q^2 d T^3 \in A_3(R,q)$ on the normal subgroup $N_{3,1}(R,q)$ yields

$$ \left [\begin{matrix}
      b    &    -a   &  a^2/b & -a^3/b^2 \\
        0     &    1   & -2a/b & 3a^2/b^2 \\
        0     &    0   &    1/b & -3a/b^2 \\
       0     &    0   &      0 &   b^{-2}  
\end{matrix}\right ]
$$

\end{ex}

\begin{ex}
 The action of $a+b + q cT^2 + q^2 d T^3 + q^3 e T^4 \in A_4(R,q)$ on the normal subgroup $N_{4,1}(R,q)$ yields 
 
\begin{equation}
\left[ \begin{matrix}
 b &        -a   &   a^2/b  & -a^3/b^2  &  a^4/b^3 \\
0   &       1    & -2a/b  &3a^2/b^2 &-4a^3/b^3 \\
0    &      0     &   1/b  & -3a/b^2 & 6a^2/b^3 \\
0     &     0     &     0  &   b^{-2}  & -4a/b^3 \\
0      &    0     &     0  &        0   &  b^{-3}
  
 \end{matrix} \right]
\end{equation}
\end{ex}

\begin{ex} The following matrix describes the action of $a+b+qcT^2 + q^2 d T^3 \in A_4(R,q)$ on $N_{2,2}(R,q)$.
\begin{equation}
 \left[\begin{array}{rrrrr}
-\frac{2 \, a c q}{b} + b & \frac{a^{2} c q}{b^{2}} - a &
\frac{a^{2}}{b} & -\frac{a^{4} c q}{b^{4}} - \frac{a^{3}}{b^{2}}
& \frac{a^{4} q}{b^{3}} \\
\frac{2 \, c q}{b} & -\frac{2 \, a c q}{b^{2}} + 1 & -\frac{2 \,
a}{b} & \frac{4 \, a^{3} c q}{b^{4}} + \frac{3 \, a^{2}}{b^{2}}
& -\frac{4 \, a^{3} q}{b^{3}} \\
0 & \frac{c q}{b^{2}} & \frac{1}{b} & -\frac{6 \, a^{2} c
q}{b^{4}} - \frac{3 \, a}{b^{2}} & \frac{6 \, a^{2} q}{b^{3}} \\
0 & 0 & 0 & \frac{4 \, a c q}{b^{4}} + \frac{1}{b^{2}} &
-\frac{4 \, a q}{b^{3}} \\
0 & 0 & 0 & -\frac{c q}{b^{4}} & \frac{q}{b^{3}}
\end{array}\right] 
\end{equation}

\end{ex}

\bibliographystyle{alpha}
\bibliography{../../bib/clean}

\end{document}